\documentclass{article}
\usepackage{amsmath}
\usepackage{amsopn}
\usepackage{amssymb}
\usepackage{latexsym}
\usepackage{amsfonts}
\usepackage{amsthm}
\usepackage[all]{xy}

\newtheorem{thm}{Theorem}[section]
\newtheorem{pro}[thm]{Proposition}
\newtheorem{lem}[thm]{Lemma}
\newtheorem{cla}[thm]{Claim}

\newtheorem{cor}[thm]{Corollary}

\theoremstyle{definition}
\newtheorem{obs}[thm]{Observation}

\newtheorem{rem}[thm]{Remark}
\newtheorem{exa}[thm]{Example}
\newtheorem{defn}[thm]{Definition}

\newtheorem{conj}[thm]{Conjecture}

\newcommand{\een}{\end{enumerate}}
\newcommand{\blem}{\begin{lem}}
\newcommand{\elem}{\end{lem}}
\newcommand{\bcl}{\begin{cla}}
\newcommand{\ecl}{\end{cla}}
\newcommand{\ethm}{\end{thm}}
\newcommand{\bpr}{\begin{pro}}
\newcommand{\epr}{\end{pro}}
\newcommand{\bco}{\begin{cor}}
\newcommand{\eco}{\end{cor}}
\newcommand{\bcon}{\begin{conj}}
\newcommand{\econ}{\end{conj}}
\newcommand{\bde}{\begin{defn}}
\newcommand{\ede}{\end{defn}}
\newcommand{\bex}{\begin{exa}}
\newcommand{\eexa}{\end{exa}}
\newcommand{\bobs}{\begin{obs}}
\newcommand{\eobs}{\end{obs}}
\newcommand{\bexe}{\begin{exe}}
\newcommand{\eexe}{\end{exe}}

\begin{document}
\title{Balanced Abelian group valued functions on directed graphs.}
\author{Yonah Cherniavsky \\
Ariel University, Israel\\
yonahch@ariel.ac.il \and Avraham Goldstein \\
City University, New-York, USA\\
avraham.goldstein.nyc@gmail.com \and Vadim E. Levit \\
Ariel University, Israel\\
levitv@ariel.ac.il}
\date{}
\maketitle
\begin{abstract}
We discuss functions from the edges and vertices of a directed graph to an Abelian group. Such functions, when the sum of their values along any cycle is zero, are called balanced and form an Abelian group. We study this group in two cases: when we allowed to walk against the direction of an edge taking the opposite value of the function and when we are not allowed to walk against the direction.
\end{abstract}

\section{Introduction.}
Let $A$ be an Abelian group with the group operation denoted by $+$ and the identity element denoted by $0$.
Let $G$ be a graph. Roughly speaking, an $A$-valued function $f$ on vertices and/or edges of $G$ is called {\it balanced} if the sum of its values along any cycle of $G$ is $0$. Our cycles are not permitted to have repeating edges.

The study of balanced functions can be conducted in three cases:
\begin{enumerate}
\item The graph $G$ is directed with the set of vertices $V$ and the set of directed edges $E$. When traveling between the vertices, we are allowed to travel with or against the direction of the edges. The value of a function $f$ on $\bar{e}$, which represents traveling the edge $e$ against its direction, is equal to $-f(e)$. In this context, when the function is defined on edges only, the pair $(G,f)$ is called a network or a directed network. In this paper we shall call this {\it the flexible case}, meaning that the direction of an edge does not forbid us to walk against it. The notion of balanced functions on edges for the flexible case, for functions taking values only on the edges, is introduced in the literature under different names. Thus, for example, in \cite{BHN} the set of such functions is exactly $Im(d)$ and in \cite{KY}, in somewhat different language, that set is referred-to as the set of consistent graphs. In a rather common terminology introduced by Zaslavsky,~\cite{Z}, a pair of graph and such a function on edges of a graph is called a gain graph.

\item The graph $G$ is directed with the set of vertices $V$ and the set of directed edges $E$, but we are only allowed to travel with the direction of the edges. In this paper we shall call this {\it the rigid case}. When $f$ takes values only on the edges then in some literature, following Sierre, the flexible case is described as a particular instance of the rigid case by introducing the set $\mathbb{E}$ as the new set of directed edges of $G$ (the cardinality of $\mathbb{E}$ is twice that of $E$), denoting by $\bar{e}\in \mathbb{E}$ the inverse of the directed edge $\bar{e}\in \mathbb{E}$ and requiring $f(\bar{e})=-f(e)$, \cite{BHN}, \cite{Sierre}.

\item The graph $G$ is undirected. The value of a function $f$ on an edge $e$ does not depend on the direction of the travel on $e$. The case of balanced functions $f:E\rightarrow \mathbb{R}$ is studied in~\cite{OI}, where these functions are called ``cycle-vanishing edge valuations". The case of balanced functions $f:E\rightarrow A$ is studied in~\cite{CGL}. The case of balanced functions $f:V\bigcup E\rightarrow A$ is first introduced and studied in~\cite{NI} and then is studied in~\cite{CGL}.
\end{enumerate}

The subject of this paper is the group structure and the relations between groups of functions, associated with the notion of balance, on a directed graph. Namely, we study the group structures of the groups of balanced functions for the flexible and the rigid cases and the relations between these two cases.

In this article we calculate the groups of balanced functions on edges, balanceable functions on vertices and balanced functions on vertices and edges of a directed graph with values in an Abelian group for both flexible and rigid cases.

Note that if we take functions with values in a non-Abelian group then the sets of the balanced function on edges, balanceable functions on vertices and balanced functions on vertices and edges do not inherit any natural group structure. However our methods still enable to find the number of these functions when the non-Abelian group is finite.

In what follows we say that a directed graph is weakly connected if its underlying undirected graph is connected.

For the basics of Graph Theory we refer to \cite{Diestel}.
\section{The flexible case.}
Let $G=(V,E)$ be a weakly connected directed graph, possibly with loops and multiple edges. Let $v,w\in V$ be two vertices connected by an edge $e$; $v$ is the origin of $e$ and $w$ is the endpoint of $e$. For $e\in E$ denote by $\bar{e}$ the same edge as $e$ but taken in the opposite direction. Thus $\bar{e}$ goes from $w$ to $v$. Let $\mathbb E=\left\{e,\bar{e}\,|\,e\in E\right\}$.
\bde\label{pathdef} A path $P$ from a vertex $x$ to a vertex $y$ is an alternating sequence $v_1,e_1,v_2,e_2,...,v_{n},e_{n}$ of vertices from $V$ and different edges from $\mathbb E$ such that $v_1=x$ and each $e_j$, for $j=1,...,n-1$, goes from $v_j$ to $v_{j+1}$ and $e_n$ goes from $v_n$ to $y$. We permit the same edge $e$ to appear in a path twice - one time along and one time against its direction, since this is regarded as using two different edges from $\mathbb E$.
\ede
We require our graphs to be weakly connected. Namely, any two different vertices of our graph $G$ can be connected by a path.
\bde A path $P$ from a vertex $x$ to itself is called a cycle.
\ede
We permit the trivial cycle, which is the empty sequence containing no vertices and no edges.
\bde A cycle is called simple if it contains every vertex at most one time. In other words $v_i\ne v_j$ if $i\ne j$.
\ede

\bde The length of a cycle is the number of its edges.
\ede

\bde A function $f:\mathbb E\rightarrow A$ such that $f(\bar e)=-f(e)$ is called {\it balanced} if the sum $f(e_1)+...+f(e_n)$ of the values of $f$ over all the edges of any cycle of $G$ is equal to $0$.
\ede
\bde The set of all the balanced functions $f:\mathbb E\rightarrow A$ is denoted by $\mathcal{HF}(\mathbb E,A)$. $\mathcal{HF}(\mathbb E,A)$ is a subgroup of the Abelian group $A^{\mathbb E}$ of all the functions from $\mathbb E$ to $A$.
\ede
\bde A function $g:V\rightarrow A$ is called {\it balanceable} if exists some $f:\mathbb E\rightarrow A$ such that $f(\bar e)=-f(e)$ and the sum of all the values $g(v_1)+f(e_1)+g(v_2)+f(e_2)+...+g(v_n)+f(e_n)$ along any cycle of $G$ is zero. We say that this function $f:\mathbb E\rightarrow A$ balances the function $g:V\rightarrow A$.
\ede
\bde The set of all the balanceable functions $g:V\rightarrow A$ is denoted by $\mathcal{BF}(V,A)$. The group $\mathcal{BF}(V,A)$ is a subgroup of the free Abelian group $A^V$ of all the functions from $V$ to $A$.
\ede
\bde A function $h:V\bigcup\mathbb E\rightarrow A$, which takes both vertices and edges of $G$ to some elements of $A$, is called balanced if $h(\bar e)=-h(e)$ and the sum of its values $h(v_1)+h(e_1)+h(v_2)+h(e_2)+...+h(v_n)+h(e_n)$ along any cycle of $G$ is zero.
\ede
\bde The set of all the balanced functions $h:V\bigcup\mathbb E\rightarrow A$ is denoted by $\mathcal{WF}(G,A)$. The group $\mathcal{WF}(G,A)$ is a subgroup of the Abelian group $A^{V\bigcup\mathbb E}$ of all the functions from $V\bigcup\mathbb E$ to $A$.
\ede
Clearly, any balanced function $f\in\mathcal{HF}(\mathbb E,A)$ can be viewed as a balanced function from $V\bigcup\mathbb E$ to $A$ which takes zero value on every vertex of $G$. Thus, we will regard $\mathcal{HF}(\mathbb E,A)$ as a subgroup of $\mathcal{WF}(V\bigcup\mathbb E,A)$.
\bpr The quotient $\mathcal{WF}(V\bigcup\mathbb E,A)/\mathcal{HF}(\mathbb E,A)$ is naturally isomorphic to $\mathcal{BF}(V,A)$.
\epr
\begin{proof} The natural isomorphism is defined by ``forgetting" the values of $h\in \mathcal{WF}(V\bigcup\mathbb E,A)$ on the edges of $G$ and regarding it just as a balanceable function from $V$ to $A$.
\end{proof}
We review some basic definitions and facts regarding Abelian groups.
\bde The order $ord(a)$ of an element $a\in A$ is the minimal positive number such that $ord(a)a=0$. If no such positive number exist we say that $ord(a)=\infty$.
\ede
\bde The set of all elements of $A$ of order $2$ is denoted by $A_2$. $A_2$ is a subgroup of $A$.
\ede
The group $\mathcal{HF}(\mathbb E,A)$ is well understood and the following fact is well known.
\bpr\label{HFL}
The group $\mathcal{HF}(\mathbb E,A)$ is isomorphic to $A^{|V|-1}$.
\epr
\begin{proof}
Select a vertex $v$ and consider the following bijection between the group of all $A$-valued functions $g$ on $V$ such that $g(v)=0$ and the group $\mathcal{HF}(\mathbb E,A)$. For any such $g$, since each edge $e\in \mathbb E$ goes from some vertex $x$ to some vertex $y$, we define $f(e)=g(y)-g(x)$. A straightforward calculation shows that $f\in\mathcal{HF}(\mathbb E,A)$. In the other direction of the bijection, for $f\in\mathcal{HF}(\mathbb E,A)$ we inductively construct the function $g$ as follows: we set $g(v)=0$; if $g(u)$ has been defined for a vertex $u$ then for every vertex $w$, for which exists some edge $e$ from $u$ to $w$, we define $g(w)=g(u)+f(e)$. Since $f\in\mathcal{HF}(\mathbb E,A)$, any two calculations of the value of $g$ on any vertex $u$ will produce the same result. Thus, our $g$ is well-defined. Obviously, the bijection, constructed above, is a group isomorphism.
\end{proof}
Now we can state and prove one of our main results.
\begin{thm}
Let $G=(V,E)$ be a weakly connected directed graph and $G'$ be its underlying undirected graph. Then:
\begin{enumerate}
\item If $G'$ is bipartite, then the group $\mathcal{WF}(V\bigcup\mathbb E,A)$ is isomorphic to $A^{|V|}$.
\item If $G'$ is not bipartite, then $\mathcal{WF}(V\bigcup\mathbb E,A)$ is isomorphic to $A_2\times A^{|V|-1}$.
\end{enumerate}
\end{thm}
\begin{proof}
If $G$ consists only of one vertex then the (1) part of our theorem is trivial. Otherwise, let us look at any one non-loop edge of $G$:
 $$\xygraph{
!{<0cm,0cm>;<1cm,0cm>:<0cm,1cm>::}
!{(0,1) }*+{a}="1"
!{(3,1) }*+{b}="2"
"1":"2"^p  }$$
The letters on the edge and the vertices denote the values of a function $h:V\bigcup\mathbb E\rightarrow A$. Assume that $h$ is balanced, i.e. $h\in\mathcal{WF}(V\bigcup\mathbb E,A)$. Then for the cycle obtained by walking along this edge and returning back along it we have the following equation:
$$
a+p+b-p=0\,\,,
$$
which immediately implies that
$$
b=-a\,.
$$
Thus $h$ must have opposite values on any two vertices of $G$ connected by an edge. Assume that $G'$ is bipartite, which implies that $G$ has no cycles of odd length. Then $h$, restricted to the edges, must be equal to some balanced function $f\in \mathcal{HF}(\mathbb E,A)$ on the edges. Now select any vertex $v\in V$. We can construct a balanced function $h$ on vertices and edges by: for any element $a\in A$ define $h(v)=a$ and then define $h$ for all the neighbors of $v$ to be $-a$ and then for all the neighbors of the neighbors of $v$ define $h$ to be $a$ and so on. Continuing this way we will assign values $a$ or $-a$ to all the vertices of $G$. Since all the cycles are of even length, we will not get a contradiction in that process. Next we choose any function $f\in \mathcal{HF}(\mathbb E,A)$ and we set $h$ on the edges to be equal to $f$. Hence, we constructed a bijection between $\mathcal{WF}(V\bigcup\mathbb E,A)$ and the group of pairs $\left\{(a,f)\,|\,a\in A,\,f\in \mathcal{HF}(\mathbb E,A)\right\}$. This bijection is obviously also a group isomorphism. And $\left\{(a,f)\,|\,a\in A,\,f\in \mathcal{HF}(\mathbb E,A)\right\}$ is isomorphic to $A^{|V|}$, since the group $\mathcal{HF}(\mathbb E,A)$ is isomorphic to $A^{|V|-1}$ by Proposition~\ref{HFL}.

Now assume that $G'$ is not bipartite, i.e., $G$ has a cycle of odd length. As we have already seen above, the values of a balanced function $h\in\mathcal{WF}(V\bigcup\mathbb E,A)$ on the vertices must be $a$ and $-a$ for some $a\in A$. But walking along a cycle of odd length greater than 1 we get that $a=-a$, i.e. $2a=0$, which exactly means that $a\in A_2$. Thus we construct a bijection between $\mathcal{WF}(V\bigcup\mathbb E,A)$ and the group of pairs $\left\{(a,f)\,|\,a\in A_2,\,f\in \mathcal{HF}(\mathbb E,A)\right\}$. This bijection is a group isomorphism. And $\left\{(a,f)\,|\,a\in A_2,\,f\in \mathcal{HF}(\mathbb E,A)\right\}$ is isomorphic to $A_2\times A^{|V|-1}$, since the group $\mathcal{HF}(\mathbb E,A)$ is isomorphic to $A^{|V|-1}$ by Proposition~\ref{HFL}. Note that if $G$ has a loop edge $e$ from a vertex $v$ to itself then both $h(v)+h(e)$ and $h(v)-h(e)$ must be $0$ since we can travel on $e$ in both directions. This implies that $h(v)=h(e)$ and $2h(v)=0$.
\end{proof}
\begin{rem} Let $G=(V,E)$ be a weakly connected directed graph and $G'$ be its underlying undirected graph. Notice that if the graph $G'$ is bipartite, then the group of balanceable functions $\mathcal{BF}(V,A)$ is isomorphic to $A$ and if $G'$ is not bipartite, then the group of balanceable functions $\mathcal{BF}(V,A)$ is isomorphic to $A_2$ - the group of involutions of $A$.
\end{rem}
\section{The rigid case.}
Let $G=(V,E)$ be a weakly connected directed graph. Recall that in this case we are allowed to walk only in the direction of an edge but not against it. It naturally changes the notion of a path and of a cycle in comparison with the flexible case.
\bde\label{pathdef} A path $P$ from a vertex $x$ to a vertex $y$ is an alternating sequence $v_1,e_1,v_2,e_2,...,v_{n},e_{n}$ of vertices from $V$ and different edges from $E$ (and not $\mathbb E$) such that $v_1=x$ and each $e_j$, for $j=1,...,n-1$, goes from $v_j$ to $v_{j+1}$ and $e_n$ goes from $v_n$ to $y$.
\ede
For example, this triangle
$$\xygraph{
!{<0cm,0cm>;<1cm,0cm>:<0cm,1cm>::}
!{(0,1) }*+{\bullet}="1"
!{(4,1) }*+{\bullet}="2"
!{(2,3) }*+{\bullet}="3"
"1":"2" "2":"3" "1":"3"}$$
is a cycle in the flexible case but not in the rigid case.

Similarly to the flexible case denote by $\mathcal{BR}(V,A)$, $\mathcal{HR}(E,A)$ and $\mathcal{WR}(V\bigcup E,A)$ the groups of balanceable functions on vertices, balanced functions on edges and balanced functions of the entire graph $G$ (vertices and edges), respectively.
\bpr
Any function on  the set of vertices is balanceable. I.e., $\mathcal{BR}(V,A)=A^V$.
\epr
\begin{proof} Let  us take a function on vertices $g:V\rightarrow A$. Define the function $h:V\bigcup E\rightarrow A$ in the following way: $h(v)=g(v)$ for any vertex $v\in V$, $h(e)=-g(v)$ for all the edges $e\in E$ which start at $v$. Obviously $h$ is a balanced function.
\end{proof}
\bde Two vertices $x$ and $y$ of $G$ are called strongly connected if exists a path $P_1$ from $x$ to $y$ and a path $P_2$ from $y$ to $x$. We also say that every vertex is strongly connected to itself.
\ede
Note: We did not require that $P_1$ and $P_2$ do not have common edges.
\bde A cycle is a path $P$ from a vertex $x$ to itself. We denote the set of all the cycles of $G$ by $C(G)$.
\ede
Note: Since the above-mentioned $P_1$ and $P_2$ paths might have common edges, $P_1$ followed by $P_2$ might not be a cycle. There could even not exist any cycle, containing both $x$ and $y$. Consider the following example:
\begin{exa} Let $V(G)=\{x,v,w,y\}$ and $E(G)=\{e_1,e_2,e_3,e_4,e_5\}$ where $e_1=(x,v)$, $e_2=(y,v)$, $e_3=(w,x)$, $e_4=(w,y)$ and $e_5=(v,w)$.
 The path $P_1=e_1,e_5,e_4$ is the only path which goes from $x$ to $y$ and the path $P_2=e_2,e_5,e_3$ is the only path which goes from $y$ to $x$. They have a common edge $e_5$. Thus, there exist no cycle, containing both $x$ and $y$.
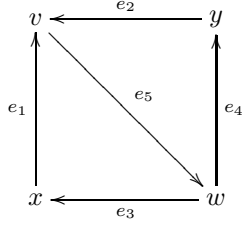
$$\xygraph{
!{<0cm,0cm>;<1.2cm,0cm>:<0cm,1.2cm>::}
!{(1,1) }*+{x}="x"
!{(1,3) }*+{v}="v"
!{(3,3) }*+{y}="y"
!{(3,1)}*+{w}="w"
"x":"v"^{e_1}   "v":"w"^{e_5}
"w":"y"_{e_4}   "w":"x"^{e_3}  "y":"v"_{e_2}
}$$
\end{exa}
Strong connectivity defines an equivalence relation on the vertices of $G$. The equivalence classes of strongly connected vertices, together with all the edges between the vertices of each class, are called the strongly connected components of $G$. We denote the number of strongly connected components of $G$ by $\bar{k}(G)$.
\begin{lem}\label{HR} If $\bar{k}(G)=1$ then the group $\mathcal{HR}(E,A)$ is isomorphic to $A^{|V|-1}$ just like in the flexible case.
\end{lem}
\begin{proof} Let $E=\{e_1,...,e_n\}$. Edge $e_1$ goes from some $x$ to some $y$. There is a path $P$ which goes from $y$ to $x$ and does not contain $e_1$, since if $P$ contains $e_1$ we can just delete this $e_1$ and all the vertices and edges, which come after it, from $P$. Thus, the sum of values of $f\in\mathcal{HR}(E,A)$ along $P$ must be equal to $-f(e_1)$. Hence we can add a new edge $\bar{e}_1$ to $G$ which goes from $y$ to $x$ and we can extend the function $f$ to a balanced function on edges of the new $G$ if and only if we set $f(\bar{e}_1)=-f(e_1)$. So the group of the balanced functions on edges of $G$ after the addition of $\bar{e}_1$ is naturally isomorphic to the original group of the balanced functions on edges of $G$ before the addition. Repeating this process for all the edges of $E$ we reduce $G$ to the flexible case, while not changing the group of the balanced functions on edges of $G$.
\end{proof}
\begin{thm}\label{directed} The group $\mathcal{HR}(E,A)$ is isomorphic to $A^{|V|-\bar{k}(G)+r(G)}$, where $r(G)$ is the number of all the edges in $G$ which go from a vertex in one strongly connected component of $G$ to a vertex in a different strongly connected component of $G$.
\end{thm}
\begin{proof} Let $V_1,...V_t$ be the equivalence classes of vertices $G$ and denote the set of edges between the vertices of $V_j$ by $E_j$. Then $\mathcal{HR}(E,A)=\mathcal{HR}(E_1,A)\times\cdots\times \mathcal{HR}(E_t,A)\times A^{U}$, where $U$ is the set of all the edges between the vertices in different strongly connected components of $G$. By Lemma~\ref{HR} we conclude that $\mathcal{HR}(E,A)$ is isomorphic to $A^{|V_1|-1+|V_2|-1+\cdots+|V_t|-1+r(G)}=A^{|V|-\bar{k}(G)+r(G)}$.
\end{proof}
\begin{thm}
$\mathcal{WR}(V\bigcup E,A)$ is isomorphic to $A^{|V|}\times  A^{|V|-\bar{k}(G)+r(G)}$
\end{thm}
\begin{proof}
To every $h\in\mathcal{WR}(V\bigcup E,A)$ corresponds the pair $(g,f)$, where $g\in \mathcal{BR}(E,A)$ is just the restriction of $h$ on vertices, and the value of $f\in \mathcal{HR}(E,A)$ on every edge $e$ is equal to $h(e)+h(v)$, where the vertex $v$ the origin of the edge $g$. Such a function $h$ is obviously a balanced function on edges since its value along any path is equal to the value of $h$ along that path. This correspondence between the elements of $\mathcal{WR}(V\bigcup E,A)$ and the pairs from $\mathcal{BR}(E,A)\times \mathcal{HR}(E,A)$ is a bijection. Indeed, for a given pair $(g,f)$, where $g$ is any function on vertices and $f$ is a balanced function on edges, we can construct $h:V\bigcup E\rightarrow A$ as follows: $h(v)=g(v)$ for all $v\in V$ and $h(e)=f(e)-g(v)$ for all $e\in E$, where the vertex $v$ is the origin of the edge $e$. The constructed bijection obviously is a group isomorphism.
\end{proof}
Thus, the flexible problem for a graph $G=(V, E)$ can be regarded as the rigid problem for the graph $G'=(V, \mathbb E)$. Vice versa, the rigid problem for a graph $G$ can be regarded as a free product of the rigid problems for the strongly connected components of $G$ also multiplied by $A^{r(G)}$ where $r(G)$ is the number of edges between different strongly connected components of $G$.

We finish this paper with the following simple claim, which connects this work to~\cite{CGL}.
\bpr Let $G$ be an undirected connected graph and let $G_{dir}$ be a directed graph obtained from $G$ by any assigning of directions to the edges of $G$. Denote by $H(E,A)$ the group of $A$-valued balanced functions on edges of $G$. Choose any order on edges of $G$ and embed $H(E,A)$ and $\mathcal{HR}(E(G_{dir}),A)$ into $A^{|E|}$. For an undirected graph $G$ the group of balanced functions on edges of $G$ is equal to the intersection of all the groups $\mathcal{HR}(E(G_{dir}),A)$, where $G_{dir}$ runs over all directed graphs for all $2^{|E|}$ possible direction assignments to the edges of $G$. The same is true for the groups of balanced functions on the entire graph (both vertices and edges). Namely, $W(V\bigcup E,A)=\bigcap\mathcal{WR}(V\bigcup E(G_{dir}),A)$.
\epr
\begin{proof}
Let $Cyc=v_1,e_1,...,v_k,e_k$ be a cycle in the undirected graph $G$. There exists a directed graph $G_{dir}$ for which $c$ is also a cycle. So any $f\in \bigcap\mathcal{HR}(E(G_{dir}),A)$ must satisfy the equation $\sum_{i=1}^kf\left(e_i\right)=0$. Therefore $f\in H(E,A)$, since $Cyc$ is an arbitrary cycle of $G$. Hence,
$$H(E,A)\supseteq\bigcap\mathcal{HR}(E(G_{dir}),A)\,.$$ The opposite inclusion is obvious, since any cycle of any $G_{dir}$ is a cycle of $G$. The proof of the second statement of the proposition is the same.
\end{proof}

\end{document}